\documentclass[12pt]{amsart}
\usepackage{amsmath,amsthm,amssymb,enumerate,mathscinet,mathtools}
\usepackage{stackengine}
\usepackage{fullpage}
\usepackage[svgnames,dvipsnames]{xcolor}
\usepackage{graphicx,xcolor,pgfplots}
\usepackage{verbatim}
\usepackage{mathrsfs}
\usepackage{hyperref}
\usepackage{enumitem}
\usepackage{blindtext}
\usepackage{multicol}
\usepackage{dsfont}
\usetikzlibrary{fit,calc,positioning,decorations.pathreplacing,matrix}
\newtheorem{theorem}{Theorem}[section]

\newtheorem{lemma}[theorem]{Lemma}

\newtheorem{proposition}[theorem]{Proposition}

\theoremstyle{definition}

\newtheorem{remark}[theorem]{Remark}

\numberwithin{equation}{section}

\newcommand{\mi}{\textsc{mi}}

\makeatletter
\newcommand*{\rom}[1]{\expandafter\@slowromancap\romannumeral #1@}

\makeatother

\def\COMMENT#1{}
\let\COMMENT=\footnote% COMMENT OUT for clean output
\allowdisplaybreaks

\title{A bijection between edges of the Tur\'an graph and irreducible elements in the dominance order lattice}

\keywords{}

\begin{document}
\author[N. Hassler]{Nathana\"el Hassler}
\address{Universit\'e Bourgogne Europe, LIB UR 7534, F-21000 Dijon, France}
\email{nathanael.hassler@ens-rennes.fr}\theoremstyle{definition}
\maketitle

 \begin{abstract}
 \noindent
In this paper we build a bijection between the meet-irreducible elements of the lattice of the compositions of $n$ with parts in $[1,p]$ equipped with the dominance order, and the edges of the $(n,p)$-Tur\'an graph. Using this bijection, we then compute asymptotically the average value of some statistics on those meet-irreducible compositions.
 \\[2mm]
 {\bf Keywords:} Tur\'an graph, bijection, dominance order, composition, meet-irreducible element.\\[2mm]
 {\bf 2020 Mathematics Subject Classification:} 05A19, 05C30, 06A07.
 \end{abstract}

\baselineskip=0.20in

\section{Introduction}
Let $n\geq 0$ and $p\geq1$ be two integers. If $a$ and $b$ are two integers, we write $a\equiv b\bmod{p}$ if $a\bmod{p}=b\bmod{p}$. The \textit{$(n,p)$-Tur\'an graph} $\mathcal{T}_n^p$ is the complete $p$-partite graph on $n$ vertices, i.e. the graph with vertex set $V(\mathcal{T}_n^p)=\{1,\ldots,n\}$, and $\{a,b\}\in E(\mathcal{T}_n^p)$ if and only if $a\not\equiv b\bmod{p}$, see Figure~\ref{fig: graph T_8^3} for an example with $(n,p)=(8,3)$. This graph is known for being the only one having the maximum number of edges on $n$ vertices and being $K_{p+1}$-free, i.e. not having a complete induced subgraph on $p+1$ vertices, see e.g. \cite{Aig}. Due to the structure of $\mathcal{T}_n^p$, one can see that its number of edges is given by (see for instance \cite[Theorem 6]{motzkin-straus})
\begin{equation}\label{closed form a_p}
    a_p(n)=\left(1-\frac{1}{p}\right)\frac{n^2}{2}-\frac{(n\bmod{p})(p-(n\bmod p))}{2p}.
\end{equation}
We refer the reader to \cite{turan,bollobas} for some standard references about the Tur\'an graph.
\begin{figure}[h]
\centering
\begin{tikzpicture}[scale=4]
\draw (-0.7,0.666)--(0,0);\draw (-0.7,0.666)--(-0.282,0);\draw (-0.7,0.666)--(0.282,0);\draw (-0.7,0.666)--(0.4,0.966);\draw (-0.7,0.666)--(0.6,0.766);
\draw (-0.5,0.866)--(0,0);\draw (-0.5,0.866)--(-0.282,0);\draw (-0.5,0.866)--(0.282,0);\draw (-0.5,0.866)--(0.4,0.966);\draw (-0.5,0.866)--(0.6,0.766);
\draw (-0.3,1.066)--(0,0);\draw (-0.3,1.066)--(-0.282,0);\draw (-0.3,1.066)--(0.282,0);\draw (-0.3,1.066)--(0.4,0.966);\draw (-0.3,1.066)--(0.6,0.766);
\draw (-0.282,0)--(0.4,0.966);\draw (-0.282,0)--(0.6,0.766);
\draw (0,0)--(0.4,0.966);\draw (0,0)--(0.6,0.766);
\draw (0.282,0)--(0.4,0.966);\draw (0.282,0)--(0.6,0.766);
\node at (0,-0.141) {$5$};
\node at (-0.6,0.966) {$4$};
\node at (-0.8,0.766) {$1$};
\node at (-0.4,1.166) {$7$};
\node at (-0.282,-0.141) {$2$};
\node at (0.282,-0.141) {$8$};
\node at (0.5,1.066) {$3$};
\node at (0.7,0.866) {$6$};
\fill[black] (0,0) circle (1pt);
\fill[black] (-0.5,0.866) circle (1pt);
\fill[black] (-0.7,0.666) circle (1pt);
\fill[black] (-0.3,1.066) circle (1pt);
\fill[black] (-0.282,0) circle (1pt);
\fill[black] (0.282,0) circle (1pt);
\fill[black] (0.4,0.966) circle (1pt);
\fill[black] (0.6,0.766) circle (1pt);
\end{tikzpicture}
\caption{The $(8,3)$-Tur\'an graph $\mathcal{T}_8^3$.}
\label{fig: graph T_8^3}
\end{figure}
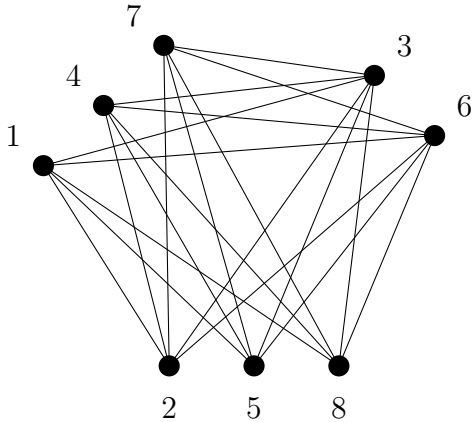

%\subsection{Meet-irreducible elements in the dominance lattice}
A \textit{composition} of $n\geq1$ is a tuple $(x_1,\ldots,x_m)$ of positive integers such that $m\geq1$ and $x_1+\ldots+x_m=n$. We denote by $\mathbb{F}_n^p$ the set of compositions of $n$ such that for all $i$, $1\leq x_i\leq p$. $\mathbb{F}_n^p$ is enumerated by the $p$-\textit{generalized Fibonacci sequence} $(F_n^p)_{n\geq0}$, defined for every $p\geq2$ by
$$F_n^p=F_{n-1}^p+F_{n-2}^p+\ldots+F_{n-p}^p$$
with initial conditions $F_n^p=0$ if $n<0$ and $F_0^p=1$ (see \cite{heu-mans}). The \textit{dominance order} on the set of compositions of $n$ is defined for two compositions $x=(x_1,\ldots,x_m)$ and $y=(y_1,\dots,y_\ell)$ by
$$x\leq y\Longleftrightarrow \mbox{ for all } 1\leq k\leq \min(m,\ell), \quad \sum_{i=1}^k x_i\leq \sum_{i=1}^k y_i,$$
see \cite{narayana fulton}. This order has mainly been studied on integer partitions, particularly for its connection to the representations of the symmetric group, see for instance~\cite{brylawski, simion, smyth}. Sapounakis, Tasoulas and Tsikouras also investigated the dominance order on Dyck paths \cite{STT}. It is proved in \cite{intfibo} that the poset $\mathbb{F}_n^p$ equipped with the dominance order forms a \textit{distributive lattice}, i.e. each pair of compositions $x,y\in\mathbb{F}_n^p$ admits a \textit{meet} (greatest lower bound) $x\wedge y$, and a \textit{join} (least upper bound) $x\vee y$, and each of the two operations distributes over the other. If $x,y\in\mathbb{F}_n^p$, we say that $y$ \textit{covers} $x$ if $x<y$ and for all $z\in\mathbb{F}_n^p$, $x\leq z\leq y\Rightarrow z\in\{x,y\}$. We say equivalently that $y$ is an \textit{upper cover} of $x$, or $x$ is a \textit{lower cover} of $y$. An element $x\in\mathbb{F}_n^p$ is \textit{meet-irreducible} if for all $y,z\in\mathbb{F}_n^p$, $x=y\wedge z\Rightarrow y=x \mbox{ or } z=x$. Dually, an element $x\in\mathbb{F}_n^p$ is \textit{join-irreducible} if for all $y,z\in\mathbb{F}_n^p$, $x=y\vee z\Rightarrow y=x \mbox{ or } z=x$. Since $\mathbb{F}_n^p$ is a finite lattice, $x$ is meet-irreducible (resp. join-irreducible) if and only if $x$ has exactly one upper cover (resp. lower cover). See for instance \cite{enum comb} for the standard definitions of lattice theory. Let $\textsc{mi}_n^p$ (resp. $\textsc{ji}_n^p$) be the set of meet-irreducible (resp. join-irreducible) elements in $\mathbb{F}_n^p$. See Figure~\ref{fig:lattice F_5^3} for an example of $\mathbb{F}_n^p$ and $\mi_n^p$ with $(n,p)=(5,3)$.

\begin{figure}[h]
\centering
    \begin{tikzpicture}[scale=1]
    \node (134) at (0,8) {$(3,2)$};
    \node (124) at (-0.866,7) {$\mathbf{(2,3)}$};
    \node (34) at (0.866,7) {$\mathbf{(3,1,1)}$};
    \node (24) at (0,6) {$(2,2,1)$};
    \node (23) at (-0.866,5) {$\mathbf{(1,3,1)}$};
    \node (14) at (0.866,5) {$\mathbf{(2,1,2)}$};
    \node (13) at (-0.866,4) {$(1,2,2)$};
    \node (4) at (0.866,4) {$\mathbf{(2,1,1,1)}$};
    \node (12) at (-0.866,3) {$\mathbf{(1,1,3)}$};
    \node (3) at (0.866,3) {$(1,2,1,1)$};
    \node (2) at (0,2) {$(1,1,2,1)$};
    \node (1) at (0,1) {$\mathbf{(1,1,1,2)}$};
    \node (0) at (0,0) {$\mathbf{(1,1,1,1,1)}$};
    \draw (134) -- (124);\draw (134) -- (34);\draw (34) -- (24);\draw (124) -- (24);\draw (24) -- (14);\draw (24) -- (23);\draw (23) -- (13);\draw (14) -- (4);\draw (14) -- (13);\draw (13) -- (12);\draw (4) -- (3);\draw (13) -- (3);\draw (3) -- (2);\draw (12) -- (2);\draw (2) -- (1);\draw (1) -- (0);
\end{tikzpicture}
\hspace{2cm}
\begin{tikzpicture}
    \node (mi) at (0,6.5) {$\mathbf{\mi_5^3}$};
    \node (124) at (0,5.5) {$(2,3)$};
    \node (34) at (0,5) {$(3,1,1)$};
    \node (23) at (0,4.5) {$(1,3,1)$};
    \node (14) at (0,4) {$(2,1,2)$};
    \node (4) at (0,3.5) {$(2,1,1,1)$};
    \node (12) at (0,3) {$(1,1,3)$};
    \node (1) at (0,2.5) {$(1,1,1,2)$};
    \node (0) at (0,2) {$(1,1,1,1,1)$};
    \node (000000) at (0,0) {$ $};
    \draw[line width=0.5mm] (-1.2,1.6)--(-1.2,7)--(1.2,7)--(1.2,1.6)--cycle;
    \draw[line width=0.5mm] (-1.2,6)--(1.2,6);
\end{tikzpicture}
    \caption{The lattice $\mathbb{F}_5^3$ and the set $\mi_5^3$ of its meet-irreducible elements.}
    \label{fig:lattice F_5^3}
\end{figure}
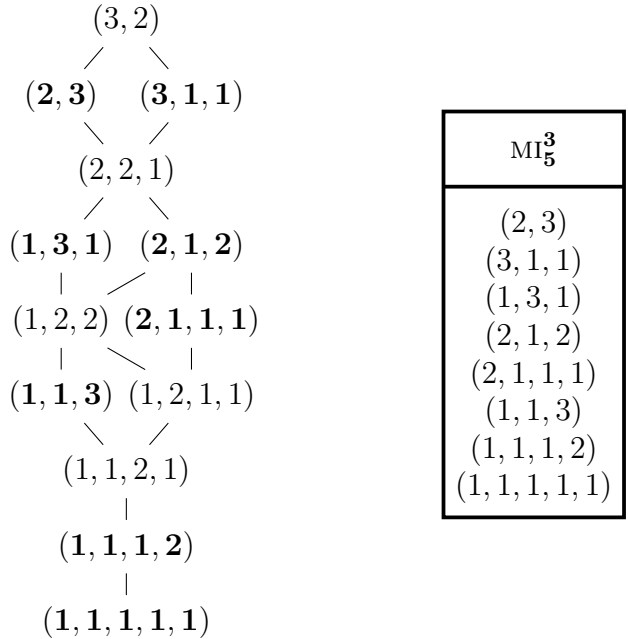

Other posets on integer compositions are studied in the literature, for instance the subword order \cite{BBD, sagan vatter, engen vatter, bjorner sagan}. However much less can be found on posets on restricted sets of compositions. In this respect, the study of $\mathbb{F}_n^p$ is natural, since its elements are enumerated by the famous $p$-generalized Fibonacci sequence, and it turns out that it presents fascinating enumerative properties. Indeed, in \cite{intfibo} the authors gave the enumeration of many characteristic elements in $\mathbb{F}_n^p$, and they proved in particular that $\mi_n^p$ is surprisingly enumerated by $a_p(n)$. As a consequence of Birkhoff's representation theorem, since $\mathbb{F}_n^p$ is a finite distributive lattice, $|\textsc{ji}_n^p|=|\mi_n^p|=a_p(n)$. The unexpected appearance of those numbers, well-known in extremal graph theory, in a context of order theory, motivates the study of the links between the corresponding objects in those two fields. In Section \ref{sec:bij}, we build an explicit bijection between $E(\mathcal{T}_n^p)$ and $\mi_n^p$, and then we show how this bijection can be adapted to $\textsc{ji}_n^p$. In Section \ref{sec:statistics} we use this bijection to compute asymptotically the average value of three statistics on $\mi_n^p$.

\section{The bijection}\label{sec:bij}

We start with the following lemma that counts the number of upper covers of a composition, and then characterizes the elements of $\mi_n^p$. For $a,b\in [1,p]$, we say that a composition $x=(x_1,\ldots,x_m)$ has an \textit{occurrence} $ab$ if there exists $1\leq i<m$ such that $x_i=a$ and $x_{i+1}=b$.

\begin{lemma}\label{nb upper covers}
    The number of upper covers of a composition $x\in\mathbb{F}_n^p$ is the number of occurrences $ab$ in $x$ with $1\leq a\leq p-1$ and $2\leq b\leq p$, plus one if $x$ has suffix $i1$ for $1\leq i\leq p-1$. 
\end{lemma}
\begin{proof}
    Suppose that $x\in\mathbb{F}_n^p$ satisfies $x=(x_1,\ldots,x_j,a,b,x_{j+1},\ldots,x_m)$ with $1\leq a\leq p-1$ and $2\leq b\leq p$. Then $z:=(x_1,\ldots,x_j,a+1,b-1,x_{j+1},\ldots,x_m)$ is an upper cover of $x$, and we say that $z$ has the form ($\star$). Similarly, if $x=(x_1,\ldots,x_m,i,1)$ with $1\leq i\leq p-1$, then $z:=(x_1\ldots,x_m,i+1)$ is an upper cover of $x$, and we say that $z$ has the form ($\star\star$). Conversely, suppose $x=(x_1,\ldots,x_m)\in\mathbb{F}_n^p$, and let $y=(y_1,\ldots,y_\ell)\in\mathbb{F}_n^p$ be such that $x<y$. To end the proof it suffices to prove that there exists $z\in\mathbb{F}_n^p$ with the form ($\star$) or ($\star\star$) such that $x<z\leq y$. Let $j$ be the smallest integer such that $x_j<y_j$. Then we have $x_j<p$.
    \vskip 5pt\noindent {Case 1:} Suppose that $x_{j+1}>1$. Then we set $z_j=x_j+1$ and $z_{j+1}=x_{j+1}-1$, and $z_i=x_i$ for $i\not\in\{j,j+1\}$. Then $z$ has the form ($\star$), and $x<z\leq y$.
    \vskip 5pt\noindent {Case 2:} Suppose that $x_{j+1}=x_{j+2}=\ldots=x_m=1$. Then we set $z_i=x_i$ if $i\leq m-2$, and $z_{m-1}=x_{m-1}+1$. Note that we may have $m=j+1$. Then $z$ has the form ($\star\star$), and $x<z\leq y$.
    \vskip 5pt\noindent {Case 3:} Suppose that $x_{j+1}=1$ and there exists $j+1<k\leq m$ such that $x_k>1$. We consider the smallest such $k$, and we set $z_{k-1}=x_{k-1}+1$ and $z_k=x_k-1$, and $z_i=x_i$ for $i\not\in\{k-1,k\}$. Then $z$ has the form ($\star$), and $x<z\leq y$.
    
    These three cases prove the converse and hence the lemma.
\end{proof}

\begin{remark}\label{nb lower covers}
    As a direct porism of Lemma \ref{nb upper covers}, the number of lower covers of $x\in\mathbb{F}_n^p$ is the number of occurrences $ab$ in $x$ with $2\leq a\leq p$ and $1\leq b\leq p-1$, plus one if $x$ does not end in 1. 
\end{remark}

Now we give a recursive decomposition of $E(\mathcal{T}_n^p)$, and then we will provide a similar decomposition for $\mi_n^p$. We have 
\begin{equation}\label{dec turan}
E(\mathcal{T}_n^p)=E(\mathcal{T}_{n-1}^p)\cup\left\{\{a,n\}\subseteq [1,n] \ | \ a\not\equiv n\bmod{p}\right\},
\end{equation}
and $$|\left\{1\leq a\leq n \ | \ a\not\equiv n\bmod{p}\right\}|=\left\lfloor\left(1-\frac{1}{p}\right) n\right\rfloor.$$
In particular, the sequence $a_p(n)$ satisfies the recurrence relation
$$a_p(n)=a_p(n-1)+\left\lfloor\left(1-\frac{1}{p}\right) n\right\rfloor.$$

In order to construct our bijection naturally, we now give a similar decomposition of $\mi_n^p$. Let $f:\mi_{n-1}^p\longrightarrow\mi_n^p$ defined for a composition $x=(x_1,\ldots,x_m)$ by
\begin{align*}
    f(x)=\left\{\begin{array}{cc}
    (x_1,\ldots,x_m+1) & \mbox{ if } 1\leq x_m<p,\\
     (x_1,\ldots,x_m,1) & \mbox{ if } x_m=p.
\end{array}\right.
\end{align*}
Then it is not hard to check that $f$ indeed takes values in $\mi_n^p$, and that $f$ is injective. Moreover, if $f(x)$ ends in 1, then it has suffix $p1$. Conversely, if $y=(y_1,\ldots,y_\ell)\in\mi_n^p$ does not have suffix $i1$ with $1\leq i\leq p-1$, then $y=f(y_1,\ldots,y_\ell-1)$ if $y_\ell>1$, and $y=f(y_1,\ldots,y_{\ell-1})$ if $(y_{\ell-1},y_\ell)=(p,1)$. Consequently, we have
\begin{equation}\label{dec mi}
\begin{array}{cl}
    \mi_n^p & =f(\mi_{n-1}^p)\cup \{x\in\mi_n^p \ | \ x \mbox{ has suffix } i1 \mbox{ with } 1\leq i\leq p-1\}  \\
     &=f(\mi_{n-1}^p)\cup \{(\underbrace{p,\ldots,p}_k,i,\underbrace{1,\ldots,1}_m)\in\mathbb{F}_n^p \ | \ 1\leq i\leq p-1,\ m\geq1,\ k\geq0\}. 
\end{array}
\end{equation}
The last equality holds since by Lemma \ref{nb upper covers}, if a composition in $\mi_n^p$ has suffix $i1$ for some $1\leq i\leq p-1$, this suffix produces one upper cover; thus such a composition avoids consecutive patterns $ab$ with $1\leq a\leq p-1$ and $2\leq b\leq p$. By the injectivity of $f$, we have 
$$|\mi_n^p|=|\mi_{n-1}^p|+|A_n^p|,$$
with $A_n^p=\{(\underbrace{p,\ldots,p}_k,i,\underbrace{1,\ldots,1}_m)\in\mathbb{F}_n^p \ | \ 1\leq i\leq p-1,\ m\geq1,\ k\geq0\}$. 
\begin{lemma}\label{card A_n^p}
    Given $n,p\geq2$, we have $|A_n^p|=\left\lfloor\left(1-\frac{1}{p}\right) n\right\rfloor$.
\end{lemma}
\begin{proof}
    For $x\in A_n^p$, let $k_x\geq0$, $1\leq i_x\leq p-1$ and $m_x\geq1$ such that $x=(\underbrace{p,\ldots,p}_{k_x},i_x,\underbrace{1,\ldots,1}_{m_x})$. Since $x$ is a composition of $n$ we have $n=k_xp+i_x+m_x$, so $m_x\equiv(n-i_x)\bmod{p}$, and $m_x\not\equiv n\bmod{p}$ since $i_x\not\equiv0\bmod{p}$. Conversely, if $1\leq a\leq n$ with $a\not\equiv n\bmod{p}$, let $y=(\underbrace{p,\ldots,p}_{\left\lfloor\frac{n-a}{p}\right\rfloor},(n-a)\bmod{p},\underbrace{1,\ldots,1}_{a})$. Since $a\not\equiv n\bmod{p}$, we have $(n-a)\bmod{p}\in\{1,\ldots,p-1\}$; therefore $y\in A_n^p$ and $m_y=a$.
    % then if $y\in A_n^p$ is such that $y=(\underbrace{p,\ldots,p}_{\left\lfloor\frac{n-a}{p}\right\rfloor},(n-a)\bmod{p},\underbrace{1,\ldots,1}_{a})$, we have $m_y=a$. 
    This proves that $x\mapsto m_x$ is a bijection between $A_n^p$ and $\{1\leq a\leq n \ | \ a\not\equiv n\bmod{p}\}$. Therefore,
    $$|A_n^p|=|\{1\leq a\leq n \ | \ a\not\equiv n\bmod{p}\}|=\left\lfloor\left(1-\frac{1}{p}\right) n\right\rfloor.$$
\end{proof}

Now for $n,p\geq2$ we give a recursive bijection $\Psi_n^p:E(\mathcal{T}_n^p)\longrightarrow \mi_n^p$ based on the decompositions from Eq. (\ref{dec turan}) and Eq.~(\ref{dec mi}), and the bijection from Lemma \ref{card A_n^p}. For $\{a,b\}\in E(\mathcal{T}_n^p)$ with $a<b$ we set
$$\Psi_n^p(a,b)=\left\{\begin{array}{cc}
    f(\Psi_{n-1}^p(a,b)) & \mbox{ if } b<n,\\
    (\underbrace{p,\ldots,p}_{\left\lfloor\frac{n-a}{p}\right\rfloor},(n-a)\bmod{p},\underbrace{1,\ldots,1}_{a}) & \mbox{ if } b=n.
\end{array}\right.$$
It follows from a direct induction that 
\begin{equation}\label{closed form Psi}
    \Psi_n^p(a,b)=(\underbrace{p,\ldots,p}_{\left\lfloor\frac{b-a}{p}\right\rfloor},(b-a)\bmod{p},\underbrace{1,\ldots,1}_{a-1},\underbrace{p,\ldots,p}_{\left\lfloor\frac{n-b+1}{p}\right\rfloor},(n-b+1)\bmod{p}),
\end{equation}
where a remainder term is omitted whenever it is zero. For $x\in\mi_n^p$, let $k_x\geq 0$, $1\leq i_x\leq p-1$ and $m_x\geq0$ maximum such that $x$ starts with $\underbrace{p,\ldots,p}_{k_x},i_x,\underbrace{1\ldots,1}_{m_x}$. Then we define $\Phi_n^p:\mi_n^p\longrightarrow E(\mathcal{T}_n^p)$ by $$\Phi_n^p(x)=\{m_x+\delta_x,p k_x+i_x+m_x+\delta_x\},$$
where 
$$\delta_x=\left\{\begin{array}{cl}
    1 & \mbox{ if } (x \mbox{ ends in }1\Rightarrow x \mbox{ has suffix }p1), \\
    0 & \mbox{ otherwise.} 
\end{array}\right.$$
Observe that $\Phi_n^p$ is well defined since $k_x,i_x$ and $m_x$ always exist for the elements of $\mi_n^p$. Moreover, $\Phi_n^p$ takes values in $E(\mathcal{T}_n^p)$ because $i_x\not\equiv0\bmod{p}$. With a direct verification, we deduce the following theorem. See Table \ref{tab:bijection} for two examples of the bijection $\Psi_n^p$.
\begin{theorem}
    For $n,p\geq2$, $\Psi_n^p$ and $\Phi_n^p$ are reciprocal bijections.
\end{theorem}

\begin{table}[h]
    \centering
    \begin{tabular}{|c|c|}
    \hline
        $\{a,b\}\in E(\mathcal{T}_7^2)$ & $\Psi_7^2(a,b)\in\mi_7^2$ \\
        \hline
        $\{1,2\}$ & $(1, 2, 2, 2)$\\
        $\{1,4\}$ & $(2, 1, 2, 2)$\\
        $\{1,6\}$ & $(2, 2, 1, 2)$\\
        $\{2,3\}$ & $(1, 1, 2, 2, 1)$\\
        $\{2,5\}$ & $(2, 1, 1, 2, 1)$\\
        $\{2,7\}$ & $(2, 2, 1, 1, 1)$\\
        $\{3,4\}$ & $(1, 1, 1, 2, 2)$\\
        $\{3,6\}$ & $(2, 1, 1, 1, 2)$\\
        $\{4,5\}$ & $(1, 1, 1, 1, 2, 1)$\\
        $\{4,7\}$ & $(2, 1, 1, 1, 1, 1)$\\
        $\{5,6\}$ & $(1, 1, 1, 1, 1, 2)$\\
        $\{6,7\}$ & $(1, 1, 1, 1, 1, 1, 1)$\\
        \hline
    \end{tabular}
    \hspace{1cm}
    \begin{tabular}{|c|c|}
    \hline
        $\{a,b\}\in E(\mathcal{T}_6^3)$ & $\Psi_6^3(a,b)\in\mi_6^3$ \\
        \hline
        $\{1,2\}$ & $(1, 3, 2)$\\
        $\{1,3\}$ & $(2, 3, 1)$\\
        $\{1,5\}$ & $(3, 1, 2)$\\
        $\{1,6\}$ & $(3, 2, 1)$\\
        $\{2,3\}$ & $(1, 1, 3, 1)$\\
        $\{2,4\}$ & $(2, 1, 3)$\\
        $\{2,6\}$ & $(3, 1, 1, 1)$\\
        $\{3,4\}$ & $(1, 1, 1, 3)$\\
        $\{3,5\}$ & $(2, 1, 1, 2)$\\
        $\{4,5\}$ & $(1, 1, 1, 1, 2)$\\
        $\{4,6\}$ & $(2, 1, 1, 1, 1)$\\
        $\{5,6\}$ & $(1, 1, 1, 1, 1, 1)$\\
        \hline
    \end{tabular}
    \vspace{0.3cm}
    \caption{Two examples of the bijection $\Psi_n^p$ for $(n,p)=(7,2)$ and $(6,3)$.}
    \label{tab:bijection}
\end{table}

\begin{remark}
    By doing the same investigation for $\textsc{ji}_n^p$, we obtain that the following map is a bijection from $E(\mathcal{T}_n^p)$ to $\textsc{ji}_n^p$:
    \begin{equation*}
    \Tilde{\Psi}_n^p(a,b)=(\underbrace{1,\ldots,1}_{a-1},(b-a)\bmod{p}+1,\underbrace{p,\ldots,p}_{\left\lfloor\frac{b-a}{p}\right\rfloor},\underbrace{1,\ldots,1}_{n-b}).
\end{equation*}
Using Remark \ref{nb lower covers}, we can check that it indeed takes values in $\textsc{ji}_n^p$. To define its reciprocal, for $x\in\mathbb{F}_n^p$, let $k_x$ (resp. $m_x$) be maximal such that $x$ starts (resp. ends) with $k_x$ (resp. $m_x$) consecutive 1s. The reciprocal of $\Tilde{\Psi}_n^p$ is then defined by
\begin{equation*}
    \Tilde{\Phi}_n^p(x)=\{k_x+1,n-m_x\}.
\end{equation*}
For $x\in\textsc{ji}_n^p$, $n-k_x-m_x-1\not\equiv 0\bmod{p}$, hence $\Tilde{\Phi}_n^p$ takes values in $E(\mathcal{T}_n^p)$. See Table \ref{tab:bijection tilde} for two examples of the bijection $\Tilde{\Psi}_n^p$.
\end{remark}

\begin{table}[h]
    \centering
    \begin{tabular}{|c|c|}
    \hline
        $\{a,b\}\in E(\mathcal{T}_7^2)$ & $\Tilde{\Psi}_7^2(a,b)\in\textsc{ji}_7^2$ \\
        \hline
        $\{1,2\}$ & $(2, 1, 1, 1, 1, 1)$\\
        $\{1,4\}$ & $(2, 2, 1, 1, 1)$\\
        $\{1,6\}$ & $(2, 2, 2, 1)$\\
        $\{2,3\}$ & $(1, 2, 1, 1, 1, 1)$\\
        $\{2,5\}$ & $(1, 2, 2, 1, 1)$\\
        $\{2,7\}$ & $(1, 2, 2, 2)$\\
        $\{3,4\}$ & $(1, 1, 2, 1, 1, 1)$\\
        $\{3,6\}$ & $(1, 1, 2, 2, 1)$\\
        $\{4,5\}$ & $(1, 1, 1, 2, 1, 1)$\\
        $\{4,7\}$ & $(1, 1, 1, 2, 2)$\\
        $\{5,6\}$ & $(1, 1, 1, 1, 2, 1)$\\
        $\{6,7\}$ & $(1, 1, 1, 1, 1, 2)$\\
        \hline
    \end{tabular}
    \hspace{1cm}
    \begin{tabular}{|c|c|}
    \hline
        $\{a,b\}\in E(\mathcal{T}_6^3)$ & $\Tilde{\Psi}_6^3(a,b)\in\textsc{ji}_6^3$ \\
        \hline
        $\{1,2\}$ & $(2, 1, 1, 1, 1)$\\
        $\{1,3\}$ & $(3, 1, 1, 1)$\\
        $\{1,5\}$ & $(2, 3, 1)$\\
        $\{1,6\}$ & $(3, 3)$\\
        $\{2,3\}$ & $(1, 2, 1, 1, 1)$\\
        $\{2,4\}$ & $(1, 3, 1, 1)$\\
        $\{2,6\}$ & $(1, 2, 3)$\\
        $\{3,4\}$ & $(1, 1, 2, 1, 1)$\\
        $\{3,5\}$ & $(1, 1, 3, 1)$\\
        $\{4,5\}$ & $(1, 1, 1, 2, 1)$\\
        $\{4,6\}$ & $(1, 1, 1, 3)$\\
        $\{5,6\}$ & $(1, 1, 1, 1, 2)$\\
        \hline
    \end{tabular}
    \vspace{0.3cm}
    \caption{Two examples of the bijection $\Tilde{\Psi}_n^p$ for $(n,p)=(7,2)$ and $(6,3)$.}
    \label{tab:bijection tilde}
\end{table}

\section{Some statistics on $\mi_n^p$}\label{sec:statistics}

In this section we use the bijection $\Psi_n^p$ to study some statistics on the set $\mi_n^p$. We illustrate the utility of this bijection by computing the average value of three statistics: $\texttt{parts}$, $\texttt{first}$ and $\texttt{wrec}$, that count respectively the number of parts, the first part and the number of weak records. The statistics $\texttt{parts}$ and $\texttt{first}$ are studied in \cite{malandro,chinn-heu} on $\mathbb{F}_n^p$, and the statistic $\texttt{wrec}$ is studied in \cite{knof-mans} on the set of all compositions. In particular, we compare our results on $\mi_n^p$ with those in \cite{malandro,chinn-heu,knof-mans} on larger sets of compositions. We start with a lemma computing some sums over $E(\mathcal{T}_n^p)$.

\begin{lemma}\label{aux sum}
    For a given $p\geq 2$ we have as $n\to\infty$
    $$\sum_{\substack{1\leq a<b\leq n\\a\not\equiv b\bmod{p}}} a=\left(1-\frac{1}{p}\right)\frac{n^3}{6}+O(n^2),$$
    $$\sum_{\substack{1\leq a<b\leq n\\a\not\equiv b\bmod{p}}} b=\left(1-\frac{1}{p}\right)\frac{n^3}{3}+O(n^2).$$
\end{lemma}
\begin{proof}
    First, let us compute the sum without the congruence constraint.
    \begin{align*}
        s(n)&:=\sum_{1\leq a<b\leq n} a =\sum_{a=1}^{n-1}\sum_{b=a+1}^{n}a=\sum_{a=1}^{n-1}(n-a)a\\
        &=\frac{n^2(n-1)}{2}-\frac{n(n-1)(2n-1)}{6}=\frac{n(n-1)(n+1)}{6}.
    \end{align*}
    Similarly,
    \begin{align*}
        S(n)&:=\sum_{1\leq a<b\leq n} b =\sum_{b=2}^{n}\sum_{a=1}^{b-1}b=\sum_{b=2}^{n}b(b-1)\\
        &=\frac{n(n+1)(2n+1)}{6}-\frac{n(n+1)}{2}=\frac{n(n-1)(n+1)}{3}.
    \end{align*}
    Now if $0\leq i\leq p-1$, let $\delta_i=1$ if $i>n\bmod{p}$, and $\delta_i=0$ otherwise. Then
    \begin{align*}
        \sum_{\substack{1\leq a<b\leq n\\a\equiv b\equiv i\bmod{p}}} a=\sum_{a=0}^{\left\lfloor n/p\right\rfloor-\delta_i-1}\sum_{b=a+1}^{\left\lfloor n/p\right\rfloor-\delta_i}(pa+i)=p\cdot s(\left\lfloor n/p\right\rfloor)+O(n^2)=\frac{n^3}{6p^2}+O(n^2).
    \end{align*}
    We deduce
    \begin{align*}
        \sum_{\substack{1\leq a<b\leq n\\a\not\equiv b\bmod{p}}} a=s(n)-\sum_{i=0}^{p-1}\sum_{\substack{1\leq a<b\leq n\\a\equiv b\equiv i\bmod{p}}} a=\frac{n^3}{6}+O(n^2)-p\left(\frac{n^3}{6p^2}+O(n^2)\right)=\left(1-\frac{1}{p}\right)\frac{n^3}{6}+O(n^2).
    \end{align*}
    Similarly, if $i=0$, $$\sum_{\substack{1\leq a<b\leq n\\a\equiv b\equiv 0\bmod{p}}} b=\sum_{b=2}^{\left\lfloor n/p\right\rfloor}\sum_{a=1}^{b-1}pb=p\cdot S(\left\lfloor n/p\right\rfloor),$$ and if $i>0$ we have
    \begin{align*}
        \sum_{\substack{1\leq a<b\leq n\\a\equiv b\equiv i\bmod{p}}} b=\sum_{b=1}^{\left\lfloor n/p\right\rfloor-\delta_i}\sum_{a=0}^{b-1}(pb+i)=p\cdot S(\left\lfloor n/p\right\rfloor)+O(n^2)=\frac{n^3}{3p^2}+O(n^2).
    \end{align*}
    Therefore
    \begin{align*}
        \sum_{\substack{1\leq a<b\leq n\\a\not\equiv b\bmod{p}}} b=S(n)-\sum_{i=0}^{p-1}\sum_{\substack{1\leq a<b\leq n\\a\equiv b\equiv i\bmod{p}}} b=\frac{n^3}{3}+O(n^2)-p\left(\frac{n^3}{3p^2}+O(n^2)\right)=\left(1-\frac{1}{p}\right)\frac{n^3}{3}+O(n^2).
    \end{align*}
\end{proof}

\subsection{The number of parts} For a composition $x=(x_1,\ldots,x_m)\in\mathbb{F}_n^p$ we denote by $\texttt{parts}(x)$ its number of parts $m$. 
\begin{proposition}
    The average number of parts of a composition in $\mi_n^p$ satisfies as $n\to\infty$
    \begin{align*}
        \frac{1}{|\mi_n^p|}\sum_{x\in\mi_n^p}\texttt{\emph{parts}}(x)\sim \left(1+\frac{2}{p}\right)\frac{n}{3}.
    \end{align*}
\end{proposition}
\begin{proof}
    From Eq. (\ref{closed form Psi}), for every $a\not\equiv b\bmod{p}$ we have 
    \begin{align*}
        \texttt{parts}(\Psi_n^p(a,b))&=\left\lfloor\frac{b-a}{p}\right\rfloor+1+a-1+\left\lfloor\frac{n-b+1}{p}\right\rfloor+\mathds{1}_{b\not\equiv n+1\bmod{p}}\\
        &=\frac{n}{p}+\left(1-\frac{1}{p}\right)a+O(1).
    \end{align*}
    The total number of parts in all compositions of $\mi_n^p$ is then
    \begin{align*}
        \sum_{x\in\mi_n^p}\texttt{parts}(x)&=\sum_{\substack{1\leq a<b\leq n\\a\not\equiv b\bmod{p}}}\texttt{parts}(\Psi_n^p(a,b))=\sum_{\substack{1\leq a<b\leq n\\a\not\equiv b\bmod{p}}}\left[\frac{n}{p}+\left(1-\frac{1}{p}\right)a\right]+O(n^2)\\
        &=a_p(n)\cdot\frac{n}{p}+\left(1-\frac{1}{p}\right)^2\frac{n^3}{6}+O(n^2),
    \end{align*}
    by Lemma \ref{aux sum}. From Eq. (\ref{closed form a_p}), $a_p(n)=\left(1-\frac{1}{p}\right)\frac{n^2}{2}+O(1)$, so we deduce
    \begin{align*}
        \frac{1}{|\mi_n^p|}\sum_{x\in\mi_n^p}\texttt{parts}(x)&=\frac{1}{a_p(n)}\left(a_p(n)\cdot\frac{n}{p}+\left(1-\frac{1}{p}\right)^2\frac{n^3}{6}+O(n^2)\right)\\
        &=\frac{n}{p}+\left(1-\frac{1}{p}\right)\frac{n}{3}+O(1)\\
        &=\left(1+\frac{2}{p}\right)\frac{n}{3}+O(1).
    \end{align*}
\end{proof}

\begin{remark}\label{rk parts}
    In \cite{malandro,chinn-heu} it is proved that the average value of $\texttt{parts}$ on $\mathbb{F}_n^p$ is asymptotically $\frac{n}{\phi_pQ_p'(\phi_p)}$, where $Q_p(x)=x^p+x^{p-1}+\ldots+x-1$, and $\phi_p$ is the unique root of $Q_p(x)$ lying in $]0,1[$.	One can check that for all $p\geq2$, we have $\frac{1}{3}\left(1+\frac{2}{p}\right)< \frac{1}{\phi_pQ_p'(\phi_p)}$. This means that asymptotically, compositions in $\mi_n^p$ have fewer parts on average than regular compositions from $\mathbb{F}_n^p$.
\end{remark}

\subsection{The first part} For a composition $x=(x_1,\ldots,x_m)\in\mathbb{F}_n^p$ we denote by $\texttt{first}(x)$ the value of its first part $x_1$. 

\begin{proposition}
The average value of the statistic \texttt{\emph{first}} on the compositions of $\mi_n^p$ satisfies as $n\to\infty$
\begin{align*}
    \frac{1}{|\mi_n^p|}\sum_{x\in\mi_n^p}\texttt{\emph{first}}(x)=p\left(1-\frac{p}{n}+\frac{p(p+1)}{3n^2}+O\left(\frac{1}{n^3}\right)\right).
\end{align*}    
\end{proposition}
\begin{proof}
    From Eq. (\ref{closed form Psi}), for every $a\not\equiv b\bmod{p}$ we have
    \begin{align*}
        \texttt{first}(\Psi_n^p(a,b))=\left\{\begin{array}{cl}
            p & \mbox{ if } b-a\geq p \\
            (b-a)\bmod{p} & \mbox{ otherwise.} 
        \end{array}\right.
    \end{align*}
    The sum of $\texttt{first}(x)$ over all compositions $x\in\mi_n^p$ is then
    \begin{align*}
        \sum_{x\in\mi_n^p}\texttt{first}(x)&=\sum_{\substack{1\leq a<b\leq n\\a\not\equiv b\bmod{p}}}\texttt{first}(\Psi_n^p(a,b))=\sum_{\substack{1\leq a<b\leq n\\a\not\equiv b\bmod{p}\\ b-a\geq p}}p+\sum_{\substack{1\leq a<b\leq n\\a\not\equiv b\bmod{p}\\ b-a<p}}(b-a)\bmod{p}.
    \end{align*}
    Let us first compute the following sum when $n\geq p-1$.
    \begin{align}
        \sum_{\substack{1\leq a<b\leq n\\a\not\equiv b\bmod{p}\\ b-a<p}}1&=\sum_{a=1}^{n-p+1}\sum_{b=a+1}^{a+p-1}1+\sum_{a=n-p+2}^{n-1}\sum_{b=a+1}^{n}1=\sum_{a=1}^{n-p+1}(p-1)+\sum_{a=n-p+2}^{n-1}(n-a)\\
        &=(n-p+1)(p-1)+\sum_{a=1}^{p-2}a=(n-p+1)(p-1)+\frac{(p-2)(p-1)}{2}.\label{nb b-a<p}
    \end{align}
    We deduce
    \begin{align*}
        \sum_{\substack{1\leq a<b\leq n\\a\not\equiv b\bmod{p}\\ b-a\geq p}}p=\sum_{\substack{1\leq a<b\leq n\\a\not\equiv b\bmod{p}}}p-\sum_{\substack{1\leq a<b\leq n\\a\not\equiv b\bmod{p}\\ b-a<p}}p=p\left(a_p(n)-(n-p+1)(p-1)-\frac{(p-2)(p-1)}{2}\right).
    \end{align*}
    Now observe that if $1\leq a<b\leq n$ with $b-a<p$, then $b-a=(b-a)\bmod{p}$. Then for $n\geq p-1$ we have
    \begin{align*}
        &\sum_{\substack{1\leq a<b\leq n\\a\not\equiv b\bmod{p}\\ b-a<p}}(b-a)\bmod{p}=\sum_{\substack{1\leq a<b\leq n\\a\not\equiv b\bmod{p}\\ b-a<p}}(b-a)=\sum_{a=1}^{n-p+1}\sum_{b=a+1}^{a+p-1}(b-a)+\sum_{a=n-p+2}^{n-1}\sum_{b=a+1}^{n}(b-a)\\
        &=\sum_{a=1}^{n-p+1}\left[\frac{(a+p)(a+p-1)-a(a+1)}{2}-(p-1)a\right]+\sum_{a=n-p+2}^{n-1}\left[\frac{n(n+1)-a(a+1)}{2}-(n-a)a\right]\\
        &=\sum_{a=1}^{n-p+1}\frac{p(p-1)}{2}+\sum_{a=n-p+2}^{n-1}\left[\frac{(n-a)^2+n-a}{2}\right]\\
        %=(n-p+1)\frac{p(p-1)}{2}+\sum_{a=1}^{p-2}\frac{a^2+a}{2}\\
        %&=(n-p+1)\frac{p(p-1)}{2}+\frac{(p-2)(p-1)p}{6}
        &=\frac{p(p-1)}{2}\left(n-p+1+\frac{p-2}{3}\right).
    \end{align*}
    Finally, 
    \begin{align*}
        \sum_{x\in\mi_n^p}\texttt{first}(x)&=p\left(a_p(n)-(n-p+1)(p-1)-\frac{(p-2)(p-1)}{2}\right)+\frac{p(p-1)}{2}\left(n-p+1+\frac{p-2}{3}\right)\\
        %&=p\left(a_p(n)-(n-p+1)\frac{p-1}{2}-\frac{(p-1)(p-2)}{3}\right)\\
        &=p\left(a_p(n)-n\cdot\frac{p-1}{2}+\frac{(p-1)(p+1)}{6}\right).
    \end{align*}
    Using that $a_p(n)=\frac{(p-1)n^2}{2p}+O(1)$, we deduce
    \begin{align*}
        \frac{1}{|\mi_n^p|}\sum_{x\in\mi_n^p}\texttt{first}(x)=p\left(1-\frac{p}{n}+\frac{p(p+1)}{3n^2}+O\left(\frac{1}{n^3}\right)\right).
    \end{align*}
\end{proof}

\begin{remark}
    It follows from \cite{malandro,chinn-heu} that the average value of $\texttt{first}$ (which, by symmetry, is the same as the average value of any part) on $\mathbb{F}_n^p$ is asymptotically $\phi_pQ_p'(\phi_p)$, as defined in Remark \ref{rk parts}. We see that on $\mi_n^p$, the symmetry no longer holds, and such compositions have on average the maximum possible part as first part.
\end{remark}

\subsection{The number of weak records} For $x=(x_1,\ldots,x_m)\in\mathbb{F}_n^p$, a \textit{weak record} is a position $1\leq j\leq m$ such that $x_i\leq x_j$ for each $1\leq i\leq j$. We denote by $\texttt{wrec}(x)$ the number of weak records of $x$. The statistic $\texttt{wrec}$ on compositions is studied in \cite{knof-mans}, here we compute its average value on $\mi_n^p$.

\begin{proposition}
The average number of weak records of a composition in $\mi_n^p$ satisfies as $n\to\infty$
    \begin{align*}
        \frac{1}{|\mi_n^p|}\sum_{x\in\mi_n^p}\texttt{\emph{wrec}}(x)\sim\frac{2n}{3p}.
    \end{align*}
\end{proposition}
\begin{proof}
    From Eq. (\ref{closed form Psi}), we have that 
    \begin{align*}
        \texttt{wrec}(\Psi_n^p(a,b))=\left\lfloor\frac{b-a}{p}\right\rfloor+\left\lfloor\frac{n-b+1}{p}\right\rfloor
    \end{align*}
    when $b-a\geq p$. We do not compute it in detail for $b-a<p$ as we shall see the contribution of those pairs $\{a,b\}$ to the total number of weak records in $\mi_n^p$ is negligible. Indeed, we saw in Eq. (\ref{nb b-a<p}) that the number of those pairs is $O(n)$, and since $\texttt{wrec}(\Psi_n^p(a,b))\leq n$ for every pair $\{a,b\}$ we have that
    \begin{align*}
        \sum_{\substack{1\leq a<b\leq n\\a\not\equiv b\bmod{p}\\ b-a<p}}\texttt{wrec}(\Psi_n^p(a,b))=O(n^2).
    \end{align*}
    Then,
    \begin{align*}
        \sum_{\substack{1\leq a<b\leq n\\a\not\equiv b\bmod{p}}}\texttt{wrec}(\Psi_n^p(a,b))&=\sum_{\substack{1\leq a<b\leq n\\a\not\equiv b\bmod{p}\\ b-a\geq p}}\left(\left\lfloor\frac{b-a}{p}\right\rfloor+\left\lfloor\frac{n-b+1}{p}\right\rfloor\right)+O(n^2)\\
        &=\frac{1}{p}\sum_{\substack{1\leq a<b\leq n\\a\not\equiv b\bmod{p}\\ b-a\geq p}}(n-a)+O(n^2).
    \end{align*}
    Using Eq. (\ref{nb b-a<p}) and Lemma \ref{aux sum}, we have
    \begin{align*}
        \sum_{\substack{1\leq a<b\leq n\\a\not\equiv b\bmod{p}\\ b-a\geq p}}a=\sum_{\substack{1\leq a<b\leq n\\a\not\equiv b\bmod{p}}}a-O(n^2)=\frac{na_p(n)}{3}-O(n^2).
    \end{align*}
    Finally,
    \begin{align*}
        \sum_{\substack{1\leq a<b\leq n\\a\not\equiv b\bmod{p}}}\texttt{wrec}(\Psi_n^p(a,b))=\frac{1}{p}\left(na_p(n)-\frac{na_p(n)}{3}\right)+O(n^2)=\frac{2na_p(n)}{3p}+O(n^2),
    \end{align*}
    which completes the proof.
\end{proof}

\begin{remark}
    In \cite{knof-mans} it is proved that the average value of $\texttt{wrec}$ on the set of all compositions of $n$ is asymptotically $\log_2{n}$, so this means that compositions from $\mi_n^p$ have on average many more weak records than regular compositions.
\end{remark}

% \section*{Acknowledgment}

% The author is grateful to the referees for their helpful and careful review. This research was funded, in part, by the Agence Nationale de la Recherche (ANR),
% grants ANR-22-CE48-0002 (PICS), ANR-25-CE48-0602 (COMETA-GAE) and by the Regional Council of Bourgogne-Franche-Comté.

%%%%%%%%%%%%%%%%%%%%%%%%%%%%%%%%%%%%%%%%%%%%%%%%%%%%%%%%%%%%%%%%%%%%

\end{document}